\documentclass[11pt]{amsart}
\usepackage[english]{babel}
\usepackage{amssymb}
\usepackage{mathtools}  
\mathtoolsset{showonlyrefs}  

\usepackage{hyperref}   

\theoremstyle{plain}
\newtheorem{corollary}{Corollary}
\theoremstyle{plain}
\newtheorem{lemma}{Lemma}
\theoremstyle{remark}
\newtheorem{remark}{Remark}
\theoremstyle{plain}
\newtheorem{theorem}{Theorem}

\begin{document}

\title{Willmore type inequality using monotonicity formulas}

\author{Xiaoxiang Chai}
\email{chaixiaoxiang@gmail.com}

\begin{abstract}
  Simon type monotonicity formulas for the Willmore functional $\int | \mathbf{H} |^2$ in the hyperbolic space $\mathbb{H}^n$ and $\mathbb{S}^n$ are obtained. The formula gives a lower bound of $\int_{\Sigma} | \mathbf{H} |^2$ where $\Sigma^2$ is any closed surface in $\mathbb{H}^n$.
\end{abstract}

{\maketitle}

\section{Introduction}

We study the Willmore functional
\begin{equation}
  \int_{\Sigma} | \mathbf{H} |^2
\end{equation}
for a 2-surface $\Sigma^2$ in the standard hyperbolic $n$-space $\mathbb{H}^n$
of constant sectional curvature -1 and standard $n$-sphere $\mathbb{S}^n$.
Here $\mathbf{H}$ is the mean curvature vector of $\Sigma$ in the respective
ambients.
Willmore established the following classic result in the Euclidean space
$\mathbb{R}^3$,

\begin{theorem}
  (See for example {\cite{willmore-mean-1971}}) Given any closed smooth
  2-surface $\Sigma^2 \subset \mathbb{R}^3$,
  \begin{equation}
    \int_{\Sigma} | \mathbf{H} |^2 \geqslant 16 \pi \label{euclidean}
  \end{equation}
  with equality occurring if and only if $\Sigma$ is a standard sphere of any
  radius.
\end{theorem}

There are many proofs of the inequality {\eqref{euclidean}}. For instance, one
can invoke classical differential geometry methods (see for example
{\cite[Theorem 4.46]{kuhnel-differential-2005}}), and Simon
{\cite{simon-existence-1993}} (cf. Gilbarg and Trudinger {\cite[Eq.
(16.31)]{gilbarg-elliptic-1983}}) obtained the inequality using a monotonicity
formula.

By a stereographic projection from $\mathbb{S}^3 \backslash \{0, 0, 0, 1\}
\subset \mathbb{R}^4$ to $\mathbb{R}^3 \times \{0\}$, one obtains an analog
result for $\mathbb{S}^3$ from the result of the Euclidean 3-space.

\begin{theorem}
  (See Introduction in {\cite{marques-min-max-2014}}) Given any closed smooth
  surface $\Sigma^2 \subset \mathbb{S}^3$,
  \begin{equation}
    \frac{1}{4} \int_{\Sigma} | \mathbf{H} |^2 \geqslant 4 \pi - | \Sigma |
    \label{sphere}
  \end{equation}
  with equality if and only if $\Sigma$ is a geodesic sphere in
  $\mathbb{S}^3$.
\end{theorem}

Chen {\cite{chen-conformal-1974}} observed the conformal invariance properties
of the Willmore functional, so his proof worked well for both the hyperbolic
$n$-space and the $n$-sphere.

\begin{theorem}
  \label{thm:hyperbolic}(See {\cite{chen-conformal-1974}}) Give any closed
  smooth 2-surface $\Sigma^2 \subset \mathbb{H}^n$,
  \begin{equation}
    \frac{1}{4} \int_{\Sigma} | \mathbf{H} |^2 \geqslant 4 \pi + | \Sigma |
    \label{hyperbolic}
  \end{equation}
  with equality occurring if and only if $\Sigma$ lies in a $\mathbb{H}^3$
  subspace as a geodesic sphere.
\end{theorem}

We are concerned with the hyperbolic case mainly, and with the notations not
clarified now we have the following result.

\begin{theorem}
  \label{thm:mono}Give any closed 2-surface $\Sigma^2 \subset \mathbb{H}^n$,
  if $o \in \Sigma$ is a point of multiplicity $k \geqslant 1$,
  \begin{equation}
    | \Sigma | + 4 k \pi = - \int_{\Sigma} \left| \frac{1}{w} X^{\bot} +
    \frac{1}{2} \mathbf{H} \right|^2 + \frac{1}{4} \int_{\Sigma} | \mathbf{H}
    |^2 . \label{mono}
  \end{equation}
\end{theorem}

It easily implies Chen's result {\cite{chen-conformal-1974}} as a corollary.
Also, based on a similar idea, we prove for a surface $\Sigma^2 \subset
\mathbb{S}^n$ the following theorem.

\begin{theorem}
  \label{mono:spheres}Given any closed $\Sigma^2 \subset \mathbb{S}^n$ and $0
  < \sigma < \rho < \pi$, if $o \in \Sigma$ is a point of multiplicity $k
  \geqslant 1$, then
  \begin{align}
    & - 2 \frac{1}{w(\rho)} \int_{\Sigma_{\rho}} \cos r + 4 k \pi - | \Sigma_{\rho}
    |\\
    = & \frac{1}{w(\rho)} \int_{\Sigma_{\rho}} X^{\bot} \cdot \mathbf{H} -
    \int_{\Sigma_{\rho}} \left| \frac{1}{w} X^{\bot} + \frac{1}{2} \mathbf{H}
    \right|^2 + \frac{1}{4} \int_{\Sigma_{\rho}} | \mathbf{H} |^2 .
    \label{crude spheres}
  \end{align}
\end{theorem}

Since our proofs are based on monotonicity formulas, one can follow the same
philosophy in {\cite{hoffman-sobolev-1974}} and generalize Theorem
\ref{thm:mono} and \ref{mono:spheres} to general Riemannian manifolds with
upper sectional curvature bounds. However, for the purpose of a clear
exposition, we deal only with the two special cases $\mathbb{H}^n$ and
$\mathbb{S}^n$.

{\bfseries{Acknowledgements.}} I would like to thank sincerely my PhD advisor
Prof. Martin Man-chun Li at the Chinese University of HK for continuous
encouragement and support.

\section{Preliminaries}

First, we recall a basic comparison theorem of sectional curvature.

\begin{lemma}
  \label{comparison}(See Theorem 27 of {\cite[Chapter
  5]{petersen-riemannian-1998}}) Assume that $(M^n, g)$ satisfies $\sec
  \leqslant K$, the metric $g$ written in geodesic polar coordinates centered
  at $x \in M$ is $d r^2 + g_r$, then
  \begin{equation}
    \ensuremath{\operatorname{Hess}}_M r \geqslant
    \frac{\ensuremath{\operatorname{sn}}_K'
    (r)}{\ensuremath{\operatorname{sn}}_K (r)} g_r . \label{hess r estimate}
  \end{equation}
  Here, $\ensuremath{\operatorname{sn}}_K$ is defined to be
  \begin{align}
    \ensuremath{\operatorname{sn}}_K (r) & = \frac{1}{\sqrt{K}} \sin (\sqrt{K}
    r) \quad \text{ if } K > 0 ;\\
    \ensuremath{\operatorname{sn}}_K (r) & = r \quad \text{ if } K = 0 ;\\
    \ensuremath{\operatorname{sn}}_K (r) & = \frac{1}{\sqrt{- K}} \sinh
    (\sqrt{- K} r) \quad \text{ if } K < 0.
  \end{align}
    If $K > 0$, the estimate on $\ensuremath{\operatorname{Hess}}_M r$ is only
  valid with $r < \frac{\pi}{\sqrt{K}}$.
\end{lemma}

Suppose that the Levi-Civita connection on $(M, g)$ is $\nabla$, then $r (y)
=\ensuremath{\operatorname{dist}}_M (x, y)$ gives a vector field $\nabla r$.
We define $X$ to be the following,
\begin{equation}
  X =\ensuremath{\operatorname{sn}}_K (r) \nabla r.
\end{equation}

Given any 2-surface $\Sigma$ in $(M, g)$, we are concerned with an estimate of
the quantity $\ensuremath{\operatorname{div}}_{\Sigma} X$.

\begin{lemma}
  \label{div X}Given any $x \in \Sigma^2 \subset M^n$ with $\sec \leqslant K$,
  \begin{equation}
    \ensuremath{\operatorname{div}}_{\Sigma} X \geqslant
    2\ensuremath{\operatorname{sn}}'_K (r) .
  \end{equation}
\end{lemma}

\begin{proof}
  Let $\{e_i \}, i = 1, 2$ be a chosen orthonormal frame spanning $T \Sigma$,
  we use the convention of summation over repeated indices, and apply Theorem
  \ref{comparison},
  \begin{align}
    \ensuremath{\operatorname{div}}_{\Sigma} X & = \langle \nabla_{e_i} X, e_i
    \rangle\\
    & = \langle \nabla_{e_i} (\ensuremath{\operatorname{sn}}_K (r) \nabla r),
    e_i \rangle\\
    & = [\ensuremath{\operatorname{sn}}_K (r) \langle \nabla_{e_i} \nabla r,
    e_i \rangle +\ensuremath{\operatorname{sn}}'_K (r) | \nabla e_i r|^2]\\
    & = [\ensuremath{\operatorname{sn}}_K
    (r)\ensuremath{\operatorname{Hess}}_M r (e_i, e_i)
    +\ensuremath{\operatorname{sn}}_K' (r) | \nabla e_i r|^2]\\
    & \geqslant \ensuremath{\operatorname{sn}}_K' (r) [g_r (e_i, e_i) + |
    \nabla e_i r|^2]\\
    & =\ensuremath{\operatorname{sn}}_K' (r) [g_r (e_i, e_i) + | \nabla e_i
    r|^2]\\
    & =\ensuremath{\operatorname{sn}}_K' (r) g (e_i, e_i)\\
    & = 2\ensuremath{\operatorname{sn}}_K' (r) .
  \end{align}
Hence the proof is concluded.
\end{proof}

Although one can obtain the same estimate by Jacobi fields or exponential maps
similar to that of {\cite{hoffman-sobolev-1974}}, using Lemma \ref{comparison}
is much more direct and convenient in our settings.

\section{Willmore type inequalities}

Our main result Theorem \ref{thm:mono} is a finer version of the Willmore type
inequality {\eqref{hyperbolic}}. In its proof we established a monotonicity
formula {\eqref{eq:1}}.

Let $\phi (r)^{- 1} = w (r) = \int_0^r \sinh t \mathrm{d} t$. Note that $w
(r)$ is a constant multiple of the volume of a $\mathbb{H}^2$-geodesic ball.
We see that $\phi' (r) = - \sinh r / w^2 $.

\begin{proof}
  Let $\nabla$ and $\nabla^{\Sigma}$ be respectively the connections on
  $\mathbb{H}^n$ and $\Sigma$. Given any number $\sigma > 0$, we define a
  cutoff version of $\phi$ as $\phi_{\sigma} (r) = \phi (\max \{\sigma, r\})$.
  
  Let $r (x) =\ensuremath{\operatorname{dist}}_{\mathbb{H}^n} (o, x)$ be the
  geodesic distance from $o$ to any point $x \in \mathbb{H}^n$ and
  $\Sigma_{\rho} = \{x \in \Sigma : r (x) < \rho\}$. We choose the Lipschitz
  vector field $Y (x) = (\phi_{\sigma} (r) - \phi (\rho))_+ X$ where $0 <
  \sigma < \rho < \infty$. Let $V (x) = \cosh r (x)
  =\ensuremath{\operatorname{sn}}_1' (r)$, we see that $\nabla V = X$. $V$ is
  called a static potential in general relativity literatures, see Chrusciel
  and Herzlich {\cite{chrusciel-mass-2003}}. We often write $\phi$ instead of
  $\phi (r)$ and similarly for other quantities.
  
  We make use of the first variation formula,
  \begin{equation}
    \int_{\Sigma} \ensuremath{\operatorname{div}}_{\Sigma} Y = - \int_{\Sigma}
    \langle Y, \mathbf{H} \rangle, \label{1st var}
  \end{equation}
  where $\mathbf{H}$ is the mean curvature vector.
  
  We calculate $\ensuremath{\operatorname{div}}_{\Sigma} Y$ first. By Lemma
  \ref{div X},
  \begin{equation}
    \ensuremath{\operatorname{div}}_{\Sigma} Y = 2 (\phi_{\sigma} - \phi
    (\rho)) V + \sinh r | \nabla^{\Sigma} r|^2 ((\phi_{\sigma} - \phi
    (\rho))_+)' .
  \end{equation}
  Note that $| \nabla^{\bot} r|$, the length of $\nabla r$ along normal
  direction of $\Sigma$ is $| \nabla^{\bot} r|^2 = 1 - | \nabla^{\Sigma}
  r|^2$. Integrating $\ensuremath{\operatorname{div}}_{\Sigma} Y$ over
  $\Sigma$ then gives
  \begin{align}
    \int_{\Sigma} \ensuremath{\operatorname{div}}_{\Sigma} Y = & - 2 \phi
    (\rho) \int_{\Sigma_{\rho}} V + 2 \phi (\sigma) \int_{\Sigma_{\sigma}} V +
    2 \int_{\Sigma_{\rho, \sigma}} \phi V\\
    & \quad + \int_{\Sigma_{\rho} \backslash \Sigma_{\sigma}} \phi' (r) \sinh
    r (1 - | \nabla^{\bot} r|^2)\\
    = & - 2 \phi (\rho) \int_{\Sigma_{\rho}} V + 2 \phi (\sigma)
    \int_{\Sigma_{\sigma}} V\\
    & \quad + | \Sigma_{\rho} \backslash \Sigma_{\sigma} | -
    \int_{\Sigma_{\rho} \backslash \Sigma_{\sigma}} \phi' (r) \sinh r |
    \nabla^{\bot} r|^2 \label{eq:1},
  \end{align}
where we have used a consequence of \ a simple calculus,
  \begin{equation}
    2 \phi V + \phi' \sinh r = 1. \label{eq:2}
  \end{equation}
  Since $\mathbf{H}$ is a vector normal to $\Sigma$,
  \begin{align}
    - \int_{\Sigma} \langle Y, \mathbf{H} \rangle = & - \int_{\Sigma_{\rho}}
    (\phi_{\sigma} (r) - \phi (\rho))_+ \sinh r \nabla r \cdot \mathbf{H}
    \label{raw}\\
    = & \phi (\rho) \int_{\Sigma_{\rho}} \sinh r \nabla^{\bot} r \cdot
    \mathbf{H} - \phi (\sigma) \int_{\Sigma_{\sigma}} \sinh r \nabla^{\bot} r
    \cdot \mathbf{H}\\
    & \quad - \int_{\Sigma_{\rho} \backslash \Sigma_{\sigma}} \phi \sinh r
    \nabla^{\bot} r \cdot \mathbf{H} \label{eq:3} .
  \end{align}
  We then have
  \begin{align}
    & - \phi \sinh r \nabla^{\bot} r \cdot \mathbf{H} + \phi' \sinh r |
    \nabla^{\bot} r|^2\\
    = & - \frac{1}{w} X^{\bot} \cdot \mathbf{H} - \frac{|X^{\bot} |^2}{w^2}\\
    = & - \left| \frac{1}{w} X^{\bot} + \frac{1}{2} \mathbf{H} \right|^2 +
    \frac{1}{4} | \mathbf{H} |^2 \label{eq:4},
  \end{align}
  where $X^{\bot} = \sinh r \nabla^{\bot} r$ is the normal component to
  $\Sigma$ of the vector field $X$.
  
  To collect {\eqref{eq:1}}, {\eqref{eq:3}} and {\eqref{eq:4}}, one has
  \begin{align}
    & - 2 \phi (\rho) \int_{\Sigma_{\rho}} V + 2 \phi (\sigma)
    \int_{\Sigma_{\sigma}} V + | \Sigma_{\rho} \backslash \Sigma_{\sigma} |\\
    = & \phi (\rho) \int_{\Sigma_{\rho}} \sinh r \nabla^{\bot} r \cdot
    \mathbf{H} - \phi (\sigma) \int_{\Sigma_{\sigma}} \sinh r \nabla^{\bot} r
    \cdot \mathbf{H}\\
    & \quad - \int_{\Sigma_{\rho} \backslash \Sigma_{\sigma}} \left|
    \frac{1}{w} X^{\bot} + \frac{1}{2} \mathbf{H} \right|^2 + \frac{1}{4}
    \int_{\Sigma_{\rho} \backslash \Sigma_{\sigma}} | \mathbf{H} |^2 .
    \label{crude}
  \end{align}
  
  Then {\eqref{crude}} is the monotonicity formula we desired. Since $\Sigma$
  is closed, letting $\rho \to + \infty$ and $\sigma \to 0$, the above greatly
  simplifies as
  \begin{equation}
    | \Sigma | + 2 \lim_{\sigma \to 0} \phi (\sigma) \int_{\Sigma_{\sigma}} V
    = - \int_{\Sigma} \left| \frac{1}{w} X^{\bot} + \frac{1}{2} \mathbf{H}
    \right|^2 + \frac{1}{4} \int_{\Sigma} | \mathbf{H} |^2 .
  \end{equation}
  Since $\Sigma_{\sigma}$ is locally Euclidean with multiplicity $k$ at the point $o \in
  \Sigma$ and $V (0) = 1$, the fact that the limit
  \begin{align}
    \lim_{\sigma \to 0} \phi (\sigma) \int_{\Sigma_{\sigma}} V & =
    \lim_{\sigma \to 0} \frac{\pi \sigma^2}{\int_0^{\sigma} \sinh t \mathrm{d}
    t} \cdot \left( \frac{1}{\pi \sigma^2} \int_{\Sigma_{\sigma}} V \right)\\
    & = 2 \pi \lim_{\sigma \to 0} \frac{| \Sigma_{\sigma} |}{\pi \sigma^2} =
    2 k \pi \label{multiplicity}
  \end{align}
exists will give finally {\eqref{mono}}.
\end{proof}
If $\Sigma$ has multiplicity greater than 1 somewhere, {\eqref{hyperbolic}} is
similar to that of Li and Yau {\cite{li-new-1982}} when a point in $\Sigma$ is
covered multiple times by the immersion. In fact, we have the following
interesting result analogous to {\cite[Theorem 6]{li-new-1982}},

\begin{corollary}
  \label{cor:criterion}Give any closed 2-surface $\Sigma^2 \subset
  \mathbb{H}^n$,
  \begin{equation}
    \frac{1}{4} \int_{\Sigma} | \mathbf{H} |^2 < | \Sigma | + 8 \pi,
    \label{criterion}
  \end{equation}
  then $\Sigma$ is embedded in $\mathbb{H}^n$.
\end{corollary}

\begin{proof}
  If $\Sigma$ is of multiplicity at least two somewhere $x \in \Sigma$, then
  by {\eqref{multiplicity}},
  \begin{equation}
    \frac{1}{4} \int_{\Sigma} | \mathbf{H} |^2 \geqslant | \Sigma | + 8 \pi,
  \end{equation}
  but this contradicts with {\eqref{criterion}}. Hence, $\Sigma$ has to be
  embedded.
\end{proof}

When $\Sigma^2 \subset \mathbb{H}^n$ is a surface with boundary, one should
apply the first variation formula with boundary,
\begin{equation}
  \int_{\Sigma} \ensuremath{\operatorname{div}}_{\Sigma} Y = - \int_{\Sigma}
  \langle Y, \mathbf{H} \rangle + \int_{\partial \Sigma} \langle Y, \eta
  \rangle, \label{1st variation with boundary}
\end{equation}
where $\eta$ is the outward pointing normal of $\partial \Sigma$ to $\Sigma$.
We prove similarly,

\begin{corollary}
  Suppose that $\Sigma$ is a 2-surface with nonempty boundary, if $o \in
  \Sigma$ is an interior point of $\Sigma$, then
  \begin{equation}
    | \Sigma | + 4 \pi \leqslant \int_{\partial \Sigma} \langle \frac{1}{w} X
    {,} \eta \rangle - \int_{\Sigma} \left| \frac{1}{w} X^{\bot} + \frac{1}{2}
    \mathbf{H} \right|^2 + \frac{1}{4} \int_{\Sigma} | \mathbf{H} |^2 ;
    \label{mono with boundary}
  \end{equation}
  and if $o \in \partial \Sigma$ is a boundary point of $\Sigma$, then
  \begin{equation}
    | \Sigma | + 2 \pi \leqslant \int_{\partial \Sigma} \langle \frac{1}{w} X
    {,} \eta \rangle - \int_{\Sigma} \left| \frac{1}{w} X^{\bot} + \frac{1}{2}
    \mathbf{H} \right|^2 + \frac{1}{4} \int_{\Sigma} | \mathbf{H} |^2 .
    \label{mono with boundary point}
  \end{equation}
\end{corollary}

We now give an estimate of $\frac{1}{4} \int \mathbf{H}^2$ in the same fashion
as {\cite[Eq. (16.31]{gilbarg-elliptic-1983}} and arrive the following.

\begin{theorem}
  \label{different estimate}Given any closed 2-surface in $\mathbb{H}^n$ and a
  real number $\rho > 0$, then
  \begin{equation}
    4 \pi + | \Sigma_{\rho} | \leqslant \frac{1}{w (\rho)}
    \int_{\Sigma_{\rho}} V + \frac{1}{4} \int_{\Sigma} | \mathbf{H} |^2 .
    \label{finer}
  \end{equation}
\end{theorem}

\begin{proof}
  The proof goes as before with some modifications. From {\eqref{1st var}},
  {\eqref{raw}}, {\eqref{eq:1}} and {\eqref{eq:2}}, and also using the
  shorthand $X = \sinh r \nabla r$ we have
  \begin{align}
    - 2 \phi (\rho) \int_{\Sigma_{\rho}} V & + 2 \phi (\sigma)
    \int_{\Sigma_{\sigma}} V + | \Sigma_{\rho} \backslash \Sigma_{\sigma} |\\
    = & - \int_{\Sigma_{\rho}} (\phi_{\sigma} - \phi (\rho))_+ X \cdot
    \mathbf{H} - \int_{\Sigma_{\rho} \backslash \Sigma_{\sigma}}
    \frac{|X^{\bot} |^2}{w^2} \label{finer1},
  \end{align}
  
  For any $x \in \Sigma_{\rho} \backslash \Sigma_{\sigma}$, we claim the
  following
  \begin{equation}
    - (\phi_{\sigma} - \phi (\rho))_+ X \cdot \mathbf{H} - \frac{|X^{\bot}
    |^2}{w^2} \leqslant \frac{1}{4} | \mathbf{H} |^2 . \label{finer2}
  \end{equation}
  Then from {\eqref{finer1}} and {\eqref{finer2}}, we have
  \begin{align}
    - 2 \phi (\rho) \int_{\Sigma_{\rho}} V & + 2 \phi (\sigma)
    \int_{\Sigma_{\sigma}} V + | \Sigma_{\rho} \backslash \Sigma_{\sigma} |\\
    = & - \int_{\Sigma_{\sigma}} (\phi_{\sigma} - \phi (\rho))_+ X \cdot
    \mathbf{H} + \frac{1}{4} \int_{\Sigma_{\rho} \backslash \Sigma_{\sigma}} |
    \mathbf{H} |^2\\
    \leqslant & - \int_{\Sigma_{\sigma}} (\phi_{\sigma} - \phi (\rho))_+ X
    \cdot \mathbf{H} + \frac{1}{4} \int_{\Sigma} | \mathbf{H} |^2 .
  \end{align}
  By letting $\sigma \to 0$, we obtain {\eqref{finer}} immediately.
  
  Now we turn to the proof of the claimed estimate {\eqref{finer2}}. Indeed,
  when $X \cdot \mathbf{H} \geqslant 0$, the inequality is trivial. When $X
  \cdot \mathbf{H} \leqslant 0$, we have similar to {\eqref{eq:4}} that
  \begin{align}
    - (\phi_{\sigma} - \phi (\rho))_+ & X \cdot \mathbf{H} - \frac{|X^{\bot}
    |^2}{w^2}\\
    = & - \left| \frac{1}{w} X^{\bot} + \frac{1}{2} \mathbf{H} \right|^2 +
    \frac{1}{4} \mathbf{H}^2 + \phi (\rho) X \cdot \mathbf{H}\\
    \leqslant & \frac{1}{4} | \mathbf{H} |^2 .
  \end{align}
  Hence the proof is concluded.
\end{proof}

Now we turn to the case of surfaces in an $n$-sphere $\mathbb{S}^n$.

\begin{proof}[Proof of Theorem \ref{mono:spheres}]
  Let $r (x) =\ensuremath{\operatorname{dist}}_{\mathbb{S}^n} (x, o),$ $X =
  \sin r \nabla r$, $\phi (r)^{- 1} = w (r) = \int_0^r \sin t \mathrm{d} t$.
  We proceed similarly as the proof of Theorem \ref{thm:mono} noting the
  relation
  \begin{equation}
    2 \phi \cos r + \phi' \sin r = - 1.
  \end{equation}
  
\end{proof}

Similar to the proof of Theorem \ref{different estimate}, we have

\begin{theorem}
  \label{different estimate spheres}Given any closed 2-surface in
  $\mathbb{S}^n$ and a real number $0 < \rho < \pi$, then
  \begin{equation}
    4 \pi - | \Sigma_{\rho} | \leqslant \frac{1}{w (\rho)}
    \int_{\Sigma_{\rho}} \cos r + \frac{1}{4} \int_{\Sigma} | \mathbf{H} |^2 .
    \label{finer spheres}
  \end{equation}
\end{theorem}

\begin{remark}
  The reason that we do not take a limit $\rho \to \infty$ is the existence of
  conjugate points in $\mathbb{S}^n$.
\end{remark}

\section{Equality case of Corollary \ref{thm:hyperbolic}}

Now we turn to discuss the equality case of Corollary \ref{thm:hyperbolic}. We
recall some basics of the hyperboloid model of the hyperbolic space
$\mathbb{H}^n$. The readers can find relevant materials in {\cite[Chapter
3]{petersen-riemannian-1998}}. $\mathbb{H}^n$ can be realized as a
pseudo-sphere in Minkowski space i.e.
\begin{equation}
  \mathbb{H}^n = \{x = (x_0, x_1, \ldots, x_n) \in \mathbb{R}^{1, n} : \langle
  x, x \rangle = - 1, x_0 > 0\},
\end{equation}
where the bilinear form $\langle \cdot, \cdot \rangle$ is defined to be
\begin{equation}
  \langle x, y \rangle = - x_0 y_0 + x_1 y_1 + \cdots + x_n y_n .
\end{equation}
The notation $\langle \cdot, \cdot \rangle$ is still used because of no
confusion caused. By differentiating the relation $\langle x, x \rangle = -
1$, we find that the tangent space $T_x \mathbb{H}^n$ at $x \in \mathbb{H}^n$
is $\{y \in \mathbb{R}^{1, n} : \langle x, y \rangle = 0\}$. By restricting
$\langle \cdot, \cdot \rangle$ to $T_x \mathbb{H}^n$, we get the standard
Riemannian metric $\langle \cdot, \cdot \rangle_{T_x \mathbb{H}^n}$ on
$\mathbb{H}^n$. The geodesic passing through $x \in \mathbb{H}^n$ with unit
velocity $z \in T_x \mathbb{H}^n$ is
\begin{equation}
  x (\rho) = x \cosh \rho + z \sinh \rho, \quad \rho \in (- \infty, \infty) .
  \label{geodesic}
\end{equation}
If we know the endpoints $x \neq y$ of a geodesic segment, we can solve $\rho
= \rho (x, y)$ and $z : = z (x, y) \in T_x \mathbb{H}^n$,
\begin{equation}
  \langle x, y \rangle = - \cosh \rho, z = \sinh^{- 1} \rho (y + x \langle x,
  y \rangle) . \label{velocity and distance}
\end{equation}
\begin{theorem}
  The equality holds in {\eqref{mono}} if and only if $\Sigma$ lies in a
  $\mathbb{H}^3$ subspace as a geodesic 2-sphere.
\end{theorem}

\begin{proof}
  One easily verifies the equality of {\eqref{mono}} for geodesic 2-spheres.
  If equality holds in {\eqref{mono}}, then $\Sigma$ is of multiplicity one at
  every point of $x \in \Sigma$ by Corollary \ref{cor:criterion}. Also, there
  is some point $y \in \Sigma$ such that $\mathbf{H} (y) \neq 0$. Most
  importantly, the mean curvature vector $\mathbf{H}$ at $y \in \Sigma$ can be
  evaluated in terms of
  \begin{equation}
    - \frac{1}{2} \mathbf{H} (y) = \frac{\sinh \rho}{w (\rho)} \tau^{\bot},
    \label{h3}
  \end{equation}
  where $\rho$ and $\tau \in T_y \mathbb{H}^n$ are respectively the length and
  the velocity vector at $y$ of the geodesic segment from from any $x \in
  \Sigma$ to $y$. Let $\{e_1, e_2 \}$ span $T_y \Sigma$, note that
  {\eqref{h3}} says that $\tau$ is a linear combination of $\{e_1, e_2,
  \mathbf{H} (y)\}$, and hence every point $x \in \Sigma$ lies in a
  $\mathbb{H}^3$ subspace of $\mathbb{H}^n$. $\mathbb{H}^3$ is totally
  geodesic in $\mathbb{H}^n$, we can then consider $n = 3$ only. Now we
  identify every point and the tangent space as elements in $\mathbb{R}^{1,
  3}$, and from {\eqref{geodesic}},
  \begin{equation}
    \tau = x \sinh \rho + z \cosh \rho .
  \end{equation}
  Let $\nu$ be the outward pointing normal of $\Sigma$ in $\mathbb{H}^3$ at
  $y$, by {\eqref{velocity and distance}},
  \begin{align}
    - \frac{1}{2} \mathbf{H} (y) & = \frac{\sinh \rho}{w (\rho)} \langle x
    \sinh \rho + z \cosh \rho, \nu \rangle \nu\\
    & = \frac{\sinh^2 \rho}{w (\rho)} \langle x, \nu \rangle \nu +
    \frac{\cosh \rho}{w (\rho)} \langle y + x \langle x, y \rangle, \nu
    \rangle \nu\\
    & = \frac{\sinh^2 \rho}{w (\rho)} \langle x, \nu \rangle \nu +
    \frac{\cosh \rho}{w (\rho)} \langle x, \nu \rangle \langle x, y \rangle
    \nu\\
    & = \frac{\sinh^2 \rho}{w (\rho)} \langle x, \nu \rangle \nu -
    \frac{\cosh^2 \rho}{w (\rho)} \langle x, \nu \rangle \nu\\
    & = - \frac{1}{w (\rho)} \langle x, \nu \rangle \nu .
  \end{align}
  Since $\mathbf{H} = - H \nu$, the mean curvature $H$ at $y$ of $\Sigma$
  immersed in $\mathbb{H}^3$ is
  \begin{equation}
    \frac{1}{2} H (y) = - \frac{\langle x, \nu \rangle}{w (\rho)} =
    \frac{\langle x, \nu \rangle}{1 - \cosh \rho} = \frac{\langle x, \nu
    \rangle}{1 + \langle x, y \rangle} . \label{beautiful identity}
  \end{equation}
  We prove now that $H (y)$ can not be less than 0. Assume on the contrary
  that $H (y) = - 2 \coth t < 0$ where $t>0$, we fix the coordinates of $\mathbb{R}^{1,
  3}$ now by setting the point $\exp_y (- t \nu)$ to be $o = (1, 0, \ldots, 0)
  \in \mathbb{R}^{1, n}$ where $\exp_y$ is the exponential map of
  $\mathbb{H}^3$ at $y$. Note that $o$ is the origin under polar coordinates
  \begin{align}
    {}[0, \infty) \times \mathbb{S}^2 & \to \mathbb{R}^{1, n}\\
    (s, \theta) & \mapsto (\cosh s, \theta \sinh s),
  \end{align}
  where $\theta \in \mathbb{S}^2 \subset \mathbb{R}^3$. We assume that $y =
  (\cosh t, \theta \sinh t)$ since the distance from $o$ to $y$ is $t$, and $x
  = (\cosh \bar{t}, \bar{\theta} \sinh \bar{t}), \bar{t} > 0$. $\nu$ is then
  $(\sinh t, \theta \cosh t)$. By inserting the values of $y, \nu$ and $x$ to
  the identity {\eqref{beautiful identity}}, we get
  \begin{equation}
    - \coth t = \frac{- \cosh \bar{t} \sinh t + \sinh \bar{t} \cosh t \theta
    \cdot \bar{\theta}}{1 - \cosh t \cosh \bar{t} + \sinh t \sinh \bar{t}
    \theta \cdot \bar{\theta}} .
  \end{equation}
  Here $\theta \cdot \bar{\theta}$ is the standard $\mathbb{R}^3$ inner
  product. This readily reduces to
  \begin{align}
    0 & = 2 \sinh \bar{t} \sinh t \cosh t \theta \cdot \bar{\theta} - \cosh
    \bar{t} \sinh^2 t\\
    & \quad + \cosh t - \cosh^2 t \cosh \bar{t}\\
    & = \sinh \bar{t} \sinh (2 t) \theta \cdot \bar{\theta} - \cosh \bar{t}
    \cosh (2 t)
  \end{align}
  which is however not possible since $\theta \cdot \bar{\theta} \leqslant 1$.
  
  So $H (y) > 0$. We can set instead $H (y) = 2 \coth t > 0$ with $t>0$. We use this $t$
  and do the same thing as before, we arrive
  \begin{equation}
    \coth t = \frac{- \cosh \bar{t} \sinh t + \sinh \bar{t} \cosh t \theta
    \cdot \bar{\theta}}{1 - \cosh t \cosh \bar{t} + \sinh t \sinh \bar{t}
    \theta \cdot \bar{\theta}},
  \end{equation}
  and finally $t = \bar{t}$. Then $\Sigma$ has to be a geodesic sphere of
  radius $t$.
\end{proof}


\begin{thebibliography}{Che74}
  \bibitem[CH03]{chrusciel-mass-2003}Piotr~T.~Chru{\'s}ciel  and  Marc
  Herzlich. {\newblock}The mass of asymptotically hyperbolic Riemannian
  manifolds. {\newblock}{\itshape{Pacific journal of mathematics}},
  212(2):231--264, 2003.{\newblock}
  
  \bibitem[Che74]{chen-conformal-1974}Bang-Yen Chen. {\newblock}Some conformal
  invariants of submanifolds and their applications.
  {\newblock}{\itshape{Boll. Un. Mat. Ital}}, 10(4):380--385, 1974.{\newblock}
  
  \bibitem[GT83]{gilbarg-elliptic-1983}David Gilbarg  and  Neil~S.~Trudinger.
  {\newblock}{\itshape{Elliptic Partial Differential Equations of Second
  Order}}. {\newblock}Springer Berlin Heidelberg, Berlin, Heidelberg,
  1983.{\newblock}
  
  \bibitem[HS74]{hoffman-sobolev-1974}David Hoffman  and  Joel Spruck.
  {\newblock}Sobolev and isoperimetric inequalities for Riemannian
  submanifolds. {\newblock}{\itshape{Communications on Pure and Applied
  Mathematics}}, 27(6):715--727, 1974.{\newblock}
  
  \bibitem[K{\"u}h05]{kuhnel-differential-2005}Wolfgang K{\"u}hnel.
  {\newblock}{\itshape{Differential Geometry: Curves - Surfaces - Manifolds,
  Second Edition}}. {\newblock}American Mathematical Society, Providence, R.I,
  2 edition  edition, 2005.{\newblock}
  
  \bibitem[LY82]{li-new-1982}Peter Li  and  Shing-Tung Yau. {\newblock}A new
  conformal invariant and its applications to the Willmore conjecture and the
  first eigenvalue of compact surfaces. {\newblock}{\itshape{Inventiones
  mathematicae}}, 69(2):269--291, 1982.{\newblock}
  
  \bibitem[MN14]{marques-min-max-2014}Fernado~Coda Marques  and  Andre Neves.
  {\newblock}Min-Max theory and the Willmore conjecture.
  {\newblock}{\itshape{Annals of Mathematics}}, 179(2):638--782,
  2014.{\newblock}
  
  \bibitem[Pet98]{petersen-riemannian-1998}Peter Petersen.
  {\newblock}{\itshape{Riemannian geometry}},  volume  171  of
  {\itshape{Graduate Texts in Mathematics}}. {\newblock}Springer-Verlag, New
  York, 1998.{\newblock}
  
  \bibitem[Sim93]{simon-existence-1993}Leon Simon. {\newblock}Existence of
  surfaces minimizing the Willmore functional.
  {\newblock}{\itshape{Communications in Analysis and Geometry}},
  1(2):281--326, 1993.{\newblock}
  
  \bibitem[Wil71]{willmore-mean-1971}T.~J.~Willmore. {\newblock}Mean Curvature
  of Riemannian Immersions. {\newblock}{\itshape{Journal of the London
  Mathematical Society}}, s2-3(2):307--310, 1971.{\newblock}
\end{thebibliography}
\end{document}